\documentclass[11pt]{article}
\usepackage{a4wide}

\usepackage[a4paper, margin=1in]{geometry}
\usepackage{amsmath,amssymb,amsthm}
\usepackage{color}
\usepackage{hyperref}
\usepackage{setspace}
\usepackage{graphicx}
\usepackage{mathtools}
\usepackage[thinc]{esdiff}
\usepackage[shortlabels]{enumitem}

\newtheorem{theorem}{Theorem}[section]
\newtheorem{remark}{Remark}[section]

\newtheorem{lemma}[theorem]{Lemma}

\newtheorem*{definition*}{Definition}

\numberwithin{equation}{section}

\def\F{\mathcal{F}}

\def\Fq{\mathbb{F}_q}

\def\A{\mathcal{A}}
\def\B{\mathcal{B}}
\def\E{\mathcal{E}}

\def\RR{\mathcal{R}}

\def \F{{\mathbb F}}
\def \Z{{\mathbb Z}}

\def\PF{\mathbb{PF}}

\def\P{\mathcal{P}}
\def\L{\mathcal{L}}

\def\Q{\mathcal{Q}}

\def\HH{\mathcal{H}}

\def\A{\mathcal{A}}

\def\G{\mathcal{G}}
\def\D{\mathcal{D}}
\def\B{\mathcal{B}}
\def\X{\mathcal{X}}
\def\Y{\mathcal{Y}}
\def\Z{\mathcal{Z}}

\title{Multi-parameter Szemer\'{e}di-Trotter-type theorems and applications in finite fields}
\author{Hung Le\thanks{University of Science, Vietnam National University, Hanoi, Vietnam.
{\sl lequanghung\_t65@hus.edu.vn}.}\and Steven Senger\thanks{Missouri State University, Springfield, Missouri, USA. {\sl stevensenger@missouristate.edu}.}\and Minh-Quan Vo\thanks{University of Science, Vietnam National University, Ho Chi Minh City, Vietnam.
{\sl mqmath0000@gmail.com}.}}

\date{}

\begin{document}

\maketitle

\begin{abstract}
We prove some novel multi-parameter point-line incidence estimates in vector spaces over finite fields. While these could be seen as special cases of higher-dimensional incidence results, they outperform their more general counterparts in those contexts. We go on to present a number of applications to illustrate their use in combinatorial problems from geometry and number theory.
\end{abstract}
\section{Introduction}

There are many problems in mathematics that can be studied using incidence theory. That is, it is often helpful to quantify how often one collection of objects intersects with another. A prototypical result is the following, which bounds the number of incidences, or ordered pairs of the form $(p, \ell),$ where the point $p$ is incident to the line $\ell$, in the plane, originally due to Szemer\' edi and Trotter, in \cite{ST83}.
\begin{theorem}\label{ST}
Let $\P$ be a set of points and let $\L$ be a set of lines, both in $\mathbb R^2$. Then, we have 
\[|\{(p,\ell)\in \P \times \L \colon p\in \ell\}| \lesssim |\P|^\frac{2}{3}|\L|^\frac{2}{3}+|\P|+|\L|.\]
\end{theorem}
There are a number of constructions that show this theorem is tight. However, in the context of vector spaces over finite fields, the problem is quite different. Proofs of Szemer\'edi-Trotter rely on the topology of the Euclidean plane, which is not present in the finite field setting. This difference can be appreciated in the work of Bourgain, Katz, and Tao \cite{BKT} and later from Vinh \cite{Vinh11}, and Stevens and de Zeeuw in \cite{SdZ}.

In this note, we consider the incidences between pairs of points and pairs of lines in products of vector spaces over finite fields. This is a type of multi-parameter question, such as those studied recently in \cite{BI, BIO, CMPS}. We prove multi-parameter versions of some of the results mentioned above, then give some example applications to dot product problems (see \cite{CEHIK} for example), as well as some additive combinatorial results.
Note that although these types of estimates can be viewed as special cases of higher-dimensional incidence problems, such as those in \cite{IPST, HTV}, the results here will typically outperform their non-multiparameter counterparts in the settings considered.




\section{Incidence theorems}
We begin by stating a collection of new multi-parameter incidence results over finite fields, each inspired by single-parameter counterparts, but with the complications that come along with the added complexity. We then demonstrate applications of these new results in a variety of settings.

Given two integers $d_1, d_2 \ge 2$, we consider the product space $\F_q^{d_1} \times \F_q^{d_2}$. We use lower-case characters to denote a point in the product space. For example, we write $x = (x_1, x_2) \in \F_q^{d_1} \times \F_q^{d_2}$. A line-pair in $\F_q^{d_1} \times \F_q^{d_2}$ is a pair of two lines of the form $(\ell_1, \ell_2)$, where $\ell_1$ is a line in $\F_q^{d_1}$ and $\ell_2$ is a line in $\F_q^{d_2}$. Consider a set of points $\P$ and a set of line-pairs $\L$, both in $\F_q^{d_1} \times \F_q^{d_2}$. An incidence between $\P$ and $\L$ is a pair consisting of a point $x=(x_1, x_2)$ in $\P$ and a line-pair $\ell=(\ell_1,\ell_2)$ in $\L$ such that $x_1 \in \ell_1$ and $x_2 \in \ell_2$. Then, we also write $x \in \ell$. We denote the number of incidences by $I(\P, \L)$. 

We first present an analogue of the well-known Cauchy-Schwarz point-line incidence bound.

\begin{theorem}
\label{CS bound}
Let $\P$ be a set of points and let $\L$ be a set of line-pairs, both in $\F_q^{d_1} \times \F_q^{d_2}$. Then, the following two estimates hold.

1) $I(\P, \L) \lesssim q^{\frac{1}{2}}|\P|^{\frac{1}{2}}|\L| + |\P|$.

2) $I(\P, \L) \lesssim |\P||\L|^{\frac{1}{2}} + |\L|$.
\end{theorem}

While this bound is standard, it can often be improved. The following result is an analogue of the Szemer\'{e}di-Trotter theorem over finite fields in the spirit of the results of Vinh in \cite{Vinh11}.
 
\begin{theorem}\label{multi-incidence}
Let $\P$ be a set of points and $\L$ be a set of line-pairs, both in $\F_q^2 \times \F_q^2$. Then, 
\[
\left| I(\P, \L) - \frac{|\P||\L|}{q^2} \right| \lesssim q^{\frac{3}{2}}\sqrt{|\P||\L|}.
\]
In addition, if $\P$ and $\L$ are multi-sets, then we have 
\[
\left| I(\P, \L)-\frac{|\P||\L|}{q^2} \right|\le q^{\frac{3}{2}}\left(\sum_{{p}\in \overline{\P}}m({p})^2\right)^{\frac{1}{2}}\left(\sum_{\ell\in \overline{\L}}m(\ell)^2\right)^{\frac{1}{2}}.
\]
Here, $\overline{X}$ is the set of distinct elements in $X$ and $|X|=\sum_{x\in \overline{X}}m(x)$.
\end{theorem}

By generalizing the proof of Theorem \ref{multi-incidence} we can obtain the following incidence estimate in $\F_q^{d_1} \times \F_q^{d_2}$.

\begin{theorem}\label{multi-point-hyperplane}
Let $\P$ be a set of points and $\HH$ be a set of hyperplane-pairs, both in $\F_q^{d_1} \times \F_q^{d_2}$ with $2 \le d_1\le d_2$. Then, we have 
\[
\left\vert I(\P, \HH)-\frac{|\P||\HH|}{q^2} \right\vert\le q^{\frac{d_1+2d_2-3}{2}} \sqrt{|\P||\HH|}.
\]
\end{theorem}

By applying Theorem \ref{multi-point-hyperplane} in the case when $d_1=d_2=3$, we recover the following bound for $I(\P, \L)$ under some special conditions.




\begin{theorem}\label{AtimesB,AL_small}
Let $\A$ and $\B$ be any two sets of points in $\F_q^2$ with $|\A| \lesssim |\B|$. Let $\L$ be a set of line-pairs in $\F_q^2 \times \F_q^2$ with only non-vertical components. Then, we have 
\[
I(\A \times \B, \L) \lesssim \frac{|\A||\B|^{\frac{1}{2}}|\L|}{q} + q^{\frac{3}{2}}\sqrt{|\A||\B||\L|}.
\]
Furthermore, if $|\A||\L| \lesssim q^{5}$, then 
\[
I(\A \times \B, \L) \lesssim q^{\frac{3}{2}}\sqrt{|\A||\B||\L|}.
\]
\end{theorem}

We conclude our main results with one more estimate, inspired by the work of Stevens and de Zeeuw in \cite{SdZ}. However, in our setting, one cannot use the robust incidence bounds of Rudnev from \cite{Rudnev}, so the proof of our result proceeds differently from there.

\begin{theorem}\label{sophieFrank}
Let $\P$ be set of $m$ point-pairs and $\L$ be set of $n$ line-pairs, both in $\Fq^2 \times \Fq^2$, such that $q^3 > n > C^{-1} q$, $n^4 \ge C' mq^3$, then 
\[
I(\P, \L) \le C n^2 \sqrt{m/q}.
\]
\end{theorem}


\section{Proofs of the main results}
\subsection{Proof of Theorem \ref{CS bound}}

We begin by proving part 1). By the Cauchy-Schwarz inequality we first have
\begin{equation}\label{CS1.1}
I(\P, \L) = \sum_{{u} \in \P} |\L_{{u}}| \le |\P|^{\frac{1}{2}} \left(\sum_{{u} \in \P} |\L_{{u}}|^2\right)^{\frac{1}{2}},    
\end{equation}
where $\L_{{u}}$ denotes the subset of $\L$ consisting of all line-pairs incident to ${u}$. 
Further, we have
\begin{align*}
\sum_{{u} \in \P} |\L_{{u}}|^2 
&= \sum_{{u} \in \P} \sum_{\ell \in \L} \sum_{\ell' \in \L} 1_{{u} \in \ell}({u}, \ell) \cdot 1_{{u} \in \ell'}({u}, \ell') \\
&= \sum_{{u} \in \P} \sum_{\ell \in \L} 1_{{u} \in \ell}({u}, \ell) + \sum_{\ell \in \L} \sum_{\substack{\ell' \in \L\\\ell'\neq \ell}} \sum_{{u} \in \P}  1_{{u} \in \ell}({u}, \ell) \cdot 1_{{u} \in \ell'}({u}, \ell') \\
&= I(\P, \L) + \sum_{\ell \in \L} \sum_{\substack{\ell' \in \L\\\ell'\neq \ell}} \sum_{{u} \in \P}  1_{{u} \in \ell}({u}, \ell) \cdot 1_{{u} \in \ell'}({u}, \ell').
\end{align*}
Now, because any two line-pairs are incident to at most $q$ common points, we are guaranteed that  
\begin{equation}\label{CS1.2}
\sum_{{u} \in \P} |\L_{{u}}|^2 \le I(\P, \L) + q \cdot \binom{|\L|}{2}.
\end{equation}
Combining \eqref{CS1.1} and \eqref{CS1.2} then yields
\[
\frac{I(\P, \L)^2}{|\P|} \le I(\P, \L) +  q \cdot \binom{|\L|}{2}.
\]
If $I(\P, \L) \ge q \cdot \binom{|\L|}{2}$, then the estimate above leads to $I(\P, \L) \lesssim |\P|$. Otherwise, it leads to 
\[
I(\P, \L) \lesssim q^{\frac{1}{2}} |\P|^{\frac{1}{2}} |\L|.
\]
We therefore conclude that
\[
I(\P, \L) \lesssim q^{\frac{1}{2}}|\P||\L|^{\frac{1}{2}} + |\P|.
\]

To prove the second bound, 2), we first note that 
\begin{equation}\label{CS2.1}
I(\P, \L) = \sum_{\ell \in \L} |\P_{\ell}| \le |\L|^{\frac{1}{2}} \left(\sum_{\ell \in \L} |\P_{\ell}|^2\right)^{\frac{1}{2}},
\end{equation}
where $\P_{\ell}$ denotes the subset of $\P$ containing all points incident to $\ell$. As there is exactly one line-pair incident to any two points in $\F_q^{d_1} \times \F_q^{d_2}$, by proceeding similarly, we get
\begin{equation}\label{CS2.2}
\sum_{\ell \in \L} |\P_{\ell}|^2 \le I(\P, \L) + \binom{|\P|}{2}.
\end{equation}
By combining \eqref{CS2.1} and \eqref{CS2.2}, we obtain that
\[
\frac{I(\P, \L)^2}{|\L|} \le I(\P, \L) + \binom{|\P|}{2}.
\]
This means either $I(\P, \L) \lesssim |\P||\L|^{\frac{1}{2}}$ or $I(\P,\L) \lesssim |\L|$, and thus we conclude that
\[
I(\P, \L) \lesssim |\P||\L|^{\frac{1}{2}} + |\L|.
\]

\begin{remark}
    Notice that in \eqref{CS1.2}, we are using the fact that two line-pairs can share at most $q$ points. However, if one knows more about the how elements of the family of line-pairs under consideration intersect, this estimate can be tightened.
\end{remark}

\subsection{Proofs of Theorem \ref{multi-incidence} and \ref{multi-point-hyperplane}}
In this section, we will remind the reader of the expander mixing lemma, a fundamental result in the study of pseudorandom graphs. To see other examples of the expander mixing lemma for this general type of proof, we refer the readers to \cite{HTV, TV2, Vinh11}, and the references contained therein. Then we will give the full proof of the main case of Theorem \ref{multi-incidence}, and explain how to modify it to prove the multi-set case. The proof of Theorem \ref{multi-point-hyperplane} is essentially identical, but with more complicated eigenvalue calculations. To ease exposition, we merely highlight some of the key steps in the latter proof, and omit many of the routine yet tedious details. 
\subsubsection{Expander mixing lemma}
To prove Theorem \ref{multi-incidence}, we will appeal to the celebrated expander mixing lemma, as well as an $L^2$ variant of it. These are used on so-called $(n,d,\lambda)$-graphs. A graph $G$ is called an $(n, d, \lambda)$-graph if it has $n$ vertices, is $d$-regular, and has the property that if you order the absolute values of the eigenvalues of its adjacency matrix, the second-largest one is no more than $\lambda$. The second-largest is non-trivial here because the largest eigenvalue of any $d$-regular graph will be $d$. We now state the standard expander mixing lemma.

\begin{lemma}[Expander mixing lemma]\label{EML}
Given an $(n, d ,\lambda)$-regular graph $G$, and two subsets of vertices $\mathcal U$ and $\mathcal V$, we have that, $e(\mathcal U,\mathcal V),$ the number of edges connecting $\mathcal U$ and $\mathcal V$ satisfies
\[\left|e(\mathcal U,\mathcal V)-\frac{d}{n}|\mathcal U||\mathcal V|\right|\leq \lambda \sqrt{|\mathcal U||\mathcal V|}.\]
\end{lemma}

When working with multi-sets, the following $L^2$ variant is often more useful. To state it, we recall some standard concepts. For two functions $f$ and $g$ on some inner product space $\mathcal X$, we denote their inner product by
\[\langle f, g \rangle \coloneqq \sum_{x\in \mathcal X} f(x)\overline{g(x)},\]
where here, $\overline{g(x)}$ denotes complex conjugation of $g(x)$. We also define the $L^2$-norm of $f$ to be
\[||f||_2 \coloneqq \langle f, f \rangle,\]
and the expected value of $f$ to be
\[\mathbb E(f) \coloneqq \frac{1}{|\mathcal X|}\sum_{x\in \mathcal X}f(x).\]

\begin{lemma}[$L^2$ Expander mixing lemma]\label{EML2}
Given an $(n, d, \lambda)$-regular graph $G$ with adjacency matrix $A$, and two functions $f$ and $g$ in $L^2$ of the vertex set. Then we have
\[\left|\langle f, Ag\rangle -dn\mathbb E(f) \mathbb E(g)\right|\leq \lambda ||f||_2||g||_2.\]
\end{lemma}
\subsubsection{Proof of Theorem \ref{multi-incidence}}
To get a handle on the incidences in this setting, we will use a popular technique of embedding our incidence structure into an appropriate projective plane, where the regularity of incidences can more clearly be quantified and manipulated. In general, the standard projective embedding in each variable. Namely we map $\F_q^2 \times \F_q^2$ into $\PF_q^2 \times \PF_q^2$ by identifying $(x, y, z, w)$ with the equivalence class of $(x, y, 1, z ,w, 1)$ modulo dilations (nonzero scalar multiples). Following this embedding, and considering how lines consist of points, this will entail that any line-pair in $\F_q^2 \times \F_q^2$ also can be represented uniquely as an equivalence class in $\PF_q^2 \times \PF_q^2$ of some element $h = (h_1, 1, h_2, 1) \in \F_q^3 \times \F_q^3$ with $h_1, h_2 \neq 0$. For each $x = (x_1, x_2) \in \F_q^2 \times \F_q^2$, we denote $[x] = ([x_1], [x_2])$ the equivalence class of $x$ in $\F_q^3 \times \F_q^3$, i.e., $[x_i] \in \PF_q^2$ is the equivalence class of $x_i$ in $\F_q^3$ for any $i = 1, 2$. 

Let $G_q^{2,2}$ denote the graph whose vertices are the points of $\PF_q^2 \times \PF_q^2$, where two vertices $[x] = ([x_1], [x_2])$ and $[y]= ([y_1], [y_2])$ are connected if and only if
\[
\langle x_1, y_1 \rangle = \langle x_2, y_2 \rangle = 0.
\]
That is the points represented by $[x_i]$ and $[y_i]$ lie on the lines represented by $[y_i]$ and $[x_i]$, respectively. It is well-known that $G_q^{2,2}$ has $n = (q^2 + q + 1)^2$ vertices and $G_q^{2,2}$ is a $k$-regular graph with $k = (q + 1)^2$. Let $A$ be the adjacency matrix of $G_q^{2,2}$. The largest eigenvalue
of $A$ is $k=(q+1)^2,$ because the largest eigenvalue of the adjacency matrix of a $k$-regular graph is $k.$ 
Since two lines in $\PF_q^2$ intersect at exactly one point, two line-pairs in $\PF_q^2 \times \PF_q^2$ intersect at either one point or $q+1$ points. Let $N([x], [y])$ denote the number of common neighbors of two vertices $[x]$ and $[y]$ in $G_q^{2,2}$. As a result,  we get
\[
N([x], [y]) = 
\begin{cases}
(q+1)^2, & \text{if $[x_1] = [y_1]$ and $[x_2] = [y_2]$}\\
q+1, & \text{if either $[x_1] = [y_1]$ or $[x_2] = [y_2]$}\\
1, & \text{if $[x_1] \neq [y_1]$ and $[x_2] \neq [y_2]$}
\end{cases}
\]
where $x = (x_1, x_2)$ and $y = (y_1, y_2)$. Therefore we can write 
\[
A^2 = AA^T = J + ((q+1)^2 - 1)I + q E,
\]
where $J$ is the $n \times n$ all ones matrix, $I$ is the $n \times n$ identity matrix, and $E$ is the adjacency matrix of some $2q(q+1)$-regular graph. Note that this is due to the fact that for each $[x] = ([x_1], [x_2])$, there are $2((q^2+q+1)-1) = 2q(q+1)$ vertices $[y] = ([y_1],[y_2])$ such that either $[x_1] = [y_1]$ or $[x_2] = [y_2]$.

Let $\mathcal B$ be the set of vertices of $G_q^{2,2}$ that represent the collection $\P$ of points in $\F_q^2 \times \F_q^2$ and $\mathcal C$ be the set of vertices of $G$ that represent the collection $\L$ of line-pairs in $\F_q^2 \times \F_q^2$. Note that $I(\P, \L)$ is exactly $e(\mathcal B, \mathcal C)$, which is the number of ordered pairs $(u, v)$ where $u \in \mathcal B$, $v \in \mathcal C$ and $uv$ is an edge of $G_q^{2,2}$. Assuming that the absolute value of each of its eigenvalues but the largest one $k=(q+1)^2$ is at most $\lambda$, then we have the following estimate from the expander mixing lemma
\begin{align*}
\left| e(\mathcal B, \mathcal C) - \frac{(q+1)^2}{n} |\mathcal B||\mathcal C| \right| \le 
\lambda \sqrt{|\mathcal B||\mathcal C|},
\end{align*}
which means
\begin{align*}
\left| I(\P, \L) - \frac{(q+1)^2}{(q^2+q+1)^2} |\P||\L| \right| \le \lambda \sqrt{|\P||\L|}.
\end{align*}

Now, assuming that $v$ is the eigenvector to the second eigenvalue of $A^2$, we have
\[
A^2 v= Jv+ ((q+1)^2-1)I v + q Ev,
\]
and so
\[
|\lambda^2||v| \leq ((q+1)^2-1)|v| + q (2(q^2+q))|v|,
\]
which leads to
\[
|\lambda|\leq \sqrt{((q+1)^2-1) + q (2(q^2+q))} = O(q^{\frac{3}{2}}).
\]
In conclusion, we have
\begin{align*}
\left| I(\P, \L) - \frac{|\P||\L|}{q^2} \right| \lesssim 
q^{\frac{3}{2}} \sqrt{|\P||\L|},
\end{align*}
as desired.

The proof of the multi-set version of the result is identical to the proof above but with the application of the $L^2$ expander mixing lemma (Lemma \ref{EML2})instead of the traditional one (see the proof of Lemma 14 in \cite{HLR}). 

\subsubsection{Sketch of proof of Theorem \ref{multi-point-hyperplane}}
To prove this generalization of the previous result, we follow the same scheme as above, but with a few clear, yet somewhat tedious modifications. In particular, after making the necessary changes, we get through the setup of the previous proof, and encounter the graph $G_q^{d_1,d_2},$ which has been studied sufficiently for our purposes. It has $n = \frac{(q^{d_1+1}-1)(q^{d_2+1}-1)}{(q-1)^2}$ vertices and $G_q^{d_1,d_2}$ and is a $k$-regular graph with $k = \frac{(q^{d_1}-1)(q^{d_2}-1)}{(q-1)^2}$. Let $A$ be its adjacency matrix. It is well known that the largest eigenvalue
of $A$ is $k=\frac{(q^{d_1}-1)(q^{d_2}-1)}{(q-1)^2}$. Since the intersection of two hyperplanes in $\PF_q^d$ is a $d-2$-subspace. Now, considering the common neighbors in $G_Q^{d_1,d_2}$ also becomes more intricate:
\[
N([x], [y]) = 
\begin{cases}
\frac{(q^{d_1}-1)(q^{d_2}-1)}{(q-1)^2}, & \text{if $[x_1] = [y_1]$ and $[x_2] = [y_2]$}\\
\frac{(q^{d_1}-1)(q^{d_2-1}-1)}{(q-1)^2}, & \text{if either $[x_1] = [y_1]$ and $[x_2] \neq [y_2]$}\\
\frac{(q^{d_1-1}-1)(q^{d_2}-1)}{(q-1)^2}, & \text{if either $[x_1] \neq [y_1]$ and $[x_2] = [y_2]$}\\
\frac{(q^{d_1-1}-1)(q^{d_2-1}-1)}{(q-1)^2}, & \text{if $[x_1] \neq [y_1]$ and $[x_2] \neq [y_2]$}
\end{cases}
\]
This leads to a more complicated eigenvalue estimate. Namely, by continuing as in the previous proof, and assuming that $v$ is the eigenvector to the second eigenvalue of $A^2$, we obtain the estimate
\begin{align*}
|\lambda^2||v| &\leq \biggl(\frac{(q^{d_1}-1)(q^{d_2}-1)}{(q-1)^2}- \frac{(q^{d_1-1}-1)(q^{d_2-1}-1)}{(q-1)^2}\\
&+\left(q^{d_2-1}\frac{(q^{d_1-1}-1)}{q-1}\right)\left(\frac{(q^{d_2+1}-1)}{q-1}-1\right)\\
&+ \left(q^{d_1-1}\frac{(q^{d_2-1}-1)}{q-1}\right)\left(\frac{(q^{d_1+1}-1)}{q-1} -1\right)\biggr)|v|,
\end{align*}
\\
which leads to
\[
|\lambda| = O(q^{(d_1+2d_2-3)/2}).
\]
The rest follows as above, by the expander mixing lemma, or the $L^2$ expander mixing lemma in the case of multi-sets.

\subsection{Proof of Theorem \ref{AtimesB,AL_small}}
We first recall the following lemma from \cite{SdZ}, which is a useful consequence of Cauchy-Schwarz.

\begin{lemma}[{\cite[Lemma 7]{SdZ}}]\label{lemma7-SdZ}
Let $\X$ and $\Y$ be finite sets, and let $\varphi \colon \X \rightarrow \Z$ be a function, where $\Y \subset \Z$. Define the set
\[
\E = \{(x, x') \in \X \times \X \colon \varphi(x) = \varphi(x')\}.
\]
Then, we have
\[
|\{(x, y) \in \X \times \Y \colon \varphi(x) = y\}| \le |\Y|^{\frac{1}{2}} |\E|^{\frac{1}{2}}.
\]
\end{lemma}

\begin{proof}[Proof of Theorem \ref{AtimesB,AL_small}]
Since $\L$ consists only of those line-pairs with non-vertical components, a line-pair from $\L$ can only be of the form
\[
(\{Y=s_1X+t_1\}, \{Y=s_2X+t_2\}),
\]
where $(s_1, t_1, s_2, t_2)$ is some quadruple from $\F_q^4$, and $\{Y=s_1X+t_1\}$ is the set of pairs $(X,Y)\in \mathbb F_q^2$ satisfying the stated equation. We use similar shorthand notation elsewhere when context is clear. Therefore, we can identify $\L$ with the set consisting of such quadruples. Then, the number of incidences between $\P$ and $\L$ is given by
\[
I(\P, \L) = |\{(x_1, x_2, y_1, y_2, s_1, t_1, s_2, t_2) \in \A \times \B \times \L \colon x_is_i+t_i=y_i\}|.
\]
To proceed next, if we define 
\[
\E \coloneqq \{(x_1, x_2, s_1, t_1, s_2, t_2, x_1', x_2', s_1', t_1', s_2', t_2') \in (\A \times \L)^2 \colon s_ix_i+t_i=s_i'x_i'+t_i'\},
\]
then an application of Lemma \ref{lemma7-SdZ} with $\varphi(x_1,x_2,s_1,s_2,t_1,t_2) = (s_1x_1+t_1, s_2x_2+t_2)$ yields
\[
I(\P, \L) \le |\B|^{\frac{1}{2}} |\E|^{\frac{1}{2}}.
\]
We next bound $|\E|$ by applying Theorem \ref{multi-point-hyperplane} with the point set $\Q$ and the plane-pair set $\RR$ given by
\[
\Q \coloneqq \{((x_1,s'_1,t'_1), (x_2,s'_2,t'_2)) \colon (x_1, x_2) \in \A, \ (s'_1,t'_1,s'_2,t'_2) \in \L\}, 
\]
\[
\RR \coloneqq \{(\{s_1X+t_1=x_1'Y+Z\}, \{s_2X+t_2=x_2'Y+Z\}) \colon (s_1,t_1,s_2,t_2) \in \L, \ (x'_1, x'_2) \in \A\}.
\]
Since $|\Q| = |\RR| = |\A||\L|$, and $d_1=d_2=3,$ this yields
\[
|\E| \le \frac{|\A|^2|\L|^2}{q^2} + q^{3} |\A| |\L|.
\]
Furthermore, one also gets
\[
I(\P, \L) \le q^{\frac{3}{2}}|\A|^{\frac{1}{2}}|\B|^{\frac{1}{2}}|\L|^{\frac{1}{2}} 
\]
provided that $|\A||\L| \lesssim q^5$.
\end{proof}

\subsection{Proof of Theorem \ref{sophieFrank}}
We begin by stating and proving a variant of Lemma 8 from \cite{SdZ} that will allow us to reduce to our problem to estimating the incidences on a subset that behaves like a Cartesian product.
\begin{lemma}\label{lem: 3.6}
Let $\P$ be a set of $m$ point-pairs and let $\L$ be a set of $n$ line-pairs, both in $\Fq^2 \times \Fq^2$, such that there are between $c_1 K$ and $c_2 K$ line-pairs in $\L$ passing through each point-pair in $\P$ for some constants $c_2 > c_1 >0$. Assume that 
\[
K \ge \frac{4n}{c_1 m},\ K\ge \frac{32(q-1)}{c_1},\ K^3 \ge \frac{2^{12} n^2 c_2 (q-1)}{c_1^3 m}, ~\mbox{and} \quad 
\frac{c_1^2 K m}{2^5 n} \ge 2c_2 (q-1).
\]
Then, there exist two point-pairs $a=(a_1,a_2),\ b=(b_1,b_2)$ in $P$, and a set 
\[
\G \subseteq \{(\rho_1,\rho_2) \in \P \colon \rho_1 \not\in \ell_{a_1,b_1}, \rho_2 \not\in \ell_{a_2,b_2}\}
\]
with $|\G| \ge c_1^4K^4 m/(2^{11} n^2) $, such that $\G$ is covered by at most $c_2K$ line-pairs from $\L$ through $a$, and by at most $c_2K$ line-pairs from $\L$ through $b$.
\end{lemma}
\begin{proof}
Consider the following subset of line-pairs from $\L$
\[
\L_1 \coloneqq \{\ell \in \L \colon |\ell \cap \P| \ge I(\P, \L)/ (2n)\}.
\]
By definition of $\L_1$, we have
\[
I(\P, \L_1)  = I(\P, \L) - \sum_{\ell \not \in \L_1} |\ell \cap \P| > I(\P, \L) - (n - |\L_1|) \cdot \frac{I(\P, \L)}{2n}  \ge \frac{I(\P, \L)}{2}.
\]
Combining this with the assumptions given in the statement of the lemma, we further get
\[
I(\P, \L_1) > \frac{I(\P, \L)}{2} \ge \frac{c_1 Km}{2}.
\]
Now, by pigeonholing, there exists a point-pair $a=(a_1,a_2)$ in $\P$ incident to at least $\frac{I(\P, \L_1)}{2m}$ line-pairs in $\L_1$, and each line-pair in $\L_1$ is incident to at least $\lfloor c_1Km/(2n)-1 \rfloor$ point-pairs in $\P$ different from $a$.
Let $\L'$ be the set containing all line-pairs in $\L_1$ incident to $a$.
Let $\Q$ be the set of point-pairs $z=(z_1,z_2)$ in $\P$ 
such that $z_1 \neq a_1, z_2\ne a_2$, and 
that $a$ and $z$ are both incident to some line-pair $\ell \in \L_1$. By the definition of $\mathcal Q,$ and perhaps over-counting some of what is removed, we have the following estimate
\begin{align*}
|\Q| 
& \ge \frac{I(\P, \L_1)}{2m} \left( \frac{I(\P, \L)}{2n} - 1\right) - \sum_{\substack{(z_1, z_2) \in \Q\\ z_1=a_1}} I(\{(z_1, z_2) \}, \L') - \sum_{\substack{(z_1, z_2) \in \Q\\ z_2=a_2}} I(\{(z_1, z_2)\}, \L')\\
& \ge \frac{I(\P, \L_1)}{2m} \left(\frac{I(\P, \L)}{2n} - 1\right) - \sum_{\substack{z_2 \in \F_q^2\\ (a_1, z_2) \in \Q}} I(\{ (a_1, z_2) \}, \L') - \sum_{\substack{z_1 \in \F_q^2\\ (z_1, a_2) \in \Q}} I(\{(z_1 ,a_2) \}, \L').
\end{align*}
Because
\[
\sum_{\substack{z_2 \in \F_q^2\\ (a_1, z_2) \in \Q}} I(\{ (a_1, z_2) \}, \L') 
= |\{(z_2, \ell) \colon (a_1, z_2) \in \Q, (\ell, \ell_{z_2 a_2}) \in \L_1  \}|
 \le (q-1) c_2K,
\]
and similarly,
\[
\sum_{\substack{z_1 \in \F_q^2\\ (z_1, a_2) \in \Q}} I(\{ (z_1, a_2) \}, \L') 
= |\{(z_1, \ell) \colon (z_1, a_2) \in \Q, (\ell, \ell_{z_1 a_1}) \in \L_1  \}| 
 \le (q-1) c_2K,
\]
we are guaranteed that
\[
|\Q| \ge \frac{I(\P, \L_1)}{2m} \left( \frac{I(\P, \L)}{2n} -1\right) - 2(q-1) c_2 K \ge \frac{c_1K}{4} \left( \frac{c_1Km}{2n} -1 \right) -2(q-1)c_2K \ge \frac{c_1^2K^2m}{2^5n},
\]
where we have used the following two estimates
\[ K \ge \frac{4n}{c_1m}, \quad \frac{c_1^2Km}{2^5n} \ge 2c_2(q-1). \]
Let us consider
\[
\L_2 \coloneqq \{ \ell \in \L \colon |\ell \cap \Q|\ge I(\Q, \L)/(2n)\}.
\]
By pigeonholing again, we can find $b=(b_1,b_2)$ in $\P$ with $a_1 \neq b_1, a_2 \neq b_2$, and 
$$
\frac{I(\Q, \L_2)}{2|\Q|} \ge \frac{I(\Q, \L)}{4|\Q|} \ge \frac{c_1K}{4}
$$ 
line-pairs in $\L_2$ passing through.
Choosing 
\[
\G \coloneqq \Q \setminus \{ (z_1,z_2)\colon z_1 \in \ell_{a_1 b_1} \text{ or } z_2 \in \ell_{a_2 b_2},
\}
\]
and $\L''$ to be the set of all line-pairs in $\L_2$ incident to $b$, we have
\[
|\G| \ge \left(\frac{c_1 K}{4} - 2q +1 \right) \left( \frac{c_1 K|Q|}{2n} -1\right) \ge \frac{c_1K}{8} \cdot \frac{c_1^3 K^3 m (q-2)}{2^6n^2(q-1)}  \ge \frac{c_1^4 K^4 m (q-2)}{2^{9}n^2 (q-1)} \ge \frac{c_1^4 K^4 m}{2^{11}n^2},
\]
where we have used the following two estimates
\[
K \ge \frac{32(q-1)}{c_1}, \quad K^3 \ge \frac{2^{12}n^2 c_2 (q-1) }{c_1^3 m}.
\]
\end{proof}

By keeping $n$ fixed and proceeding by induction on $m$, we will prove that there exists a constant $C>0$ satisfying the theorem. When $m \le q$, we have $I(\P, \L) \le nm < C n^2 \sqrt{m/q}$, given that $n > C^{-1} q$. Suppose that the theorem is true for all $m \le M$ for some $M \ge q$. We now prove that the inequality holds with $m = M +1$. 
We argue by contradiction. Specifically, assume that 
\begin{equation}\label{contradict}
I(\P, \L) \ge C n^2 \sqrt{m/q}.
\end{equation}
Set $I\coloneqq I(\P, \L)$ and $K\coloneqq I/m$. Let us introduce the following two subsets of $\P$, depending on positive constants, $c_1$ and $c_2$, to be chosen later.
\[
\D \coloneqq  \{p \in \P\colon \mbox{there are $\le c_1 K$ line-pairs incident to $p$}\},
\]
\[
\E \coloneqq \{p \in \P\colon \mbox{there are $\ge c_2 K$ line-pairs incident to $p$} \}.
\]
Then, we see that 
\begin{equation}\label{I(D,L)}
I(\D, \L) \le |\D||\L| \le c_1 Km = c_1 I,    
\end{equation}
and that
\[
|\E| \cdot c_2 K \le  I(\E, \L) \le I \le Km,
\]
or $|\E| \le c_2^{-1} m \le M$. Then, by the induction hypothesis, we then get
\[
I(\E, \L) < C n^2 \sqrt{|\E|/q} \le \sqrt{1/c_2}  \cdot C n^2 \sqrt{m/q}.
\]
Combining this with \eqref{contradict} yields 
\begin{equation}\label{I(E,L)}
I(\E, \L) < c_2^{-\frac{1}{2}} I.    
\end{equation}
Consider the point set $\A \coloneqq \P \setminus (\D \cup \E)$. We claim that $|\A|\ge c_2^{-1}m$. Indeed, if not, then by the induction hypothesis, we again have
\[
I(\A,\L) <  C n^2 \sqrt{|\A|/q} \le \sqrt{1/c_2}  \cdot C n^2 \sqrt{m/q} < c_2^{-\frac{1}{2}}I, 
\]
which leads to
\[
I = I(\D, \L) + I(\E, \L) + I(\A, \L) < \left(c_1 + c_2^{-\frac{1}{2}} + c_2^{-1}\right) I,
\]
a contradiction when $c_1$ is small and $c_2$ is large enough. Therefore, we should have $|\A|\ge c_2^{-1}m$, and furthermore, 
\[
I(\A, \L) \ge \left(1 - c_1 - \sqrt{1/c_2}\right) I.
\]
Let $\A^1 \coloneqq \A$. 
We iteratively choose $\G^i \subset \A^i$ as in Lemma \ref{lem: 3.6}, so that there exist distinct point-pairs $a^i, b^i$ such that $\G^i$ is covered by at most $c_2 K$ line-pairs from $\L$ passing through $a^i$, and by at most $c_2 K$ line-pairs from $\L$ passing through $b^i$. We set $\A^{i+1} = \A^i \setminus \G^i $, repeat, and then terminate this process at $s$-th step when $|\A_{i+1}| \le c_2^{-1} m$. 
Note that it is a straightforward calculation to show that throughout the process, the conditions 
\[
K \ge \frac{4n}{c_1 |\A^i|},\ K\ge \frac{32(q-1)}{c_1},\ K^3 \ge \frac{2^{12} n^2 c_2 (q-1)}{c_1^3 |\A^i|}, ~\mbox{and} \quad \frac{c_1^2 K |\A^i|}{2^5 n} \ge 2c_2 (q-1)
\]
of Lemma \ref{lem: 3.6} hold if $C$ is chosen sufficiently large with $c_1 =2^{-10},\, c_2 =2^{20}$.
\begin{itemize}
\item The first inequality holds because 
\[
K = \frac{I}{m} \ge \frac{Cn^2}{\sqrt{mq}} \ge \frac{4n}{c_1c_2^{-1}m} \ge \frac{4n}{c_1|\A^i|}
\]
noting that $|\A^i| > c_2^{-1}m$ and that
\[
\frac{4\sqrt{q}}{c_1c_2^{-1} \sqrt{m}} < \frac{4}{c_1c_2^{-1}} < q < Cn,
\]
given that $q$ is sufficiently large.
\item The second and third inequalities hold because 
\[
K \ge \frac{Cn^2}{\sqrt{mq}} = \frac{Cn^2q}{\sqrt{mq^3}} \ge \frac{32(q-1)}{c_1},
\]
\[
K^3 \ge \frac{C^3n^6}{\sqrt{m^3q^3}} \ge \frac{2^{12}n^2c_2(q-1)}{c_1^3 c_2^{-1} m} \ge \frac{2^{12} n^2 c_2 (q-1)}{c_1^3 |\A^i|}
\]
noting that we have
\[
n^4 > \frac{2^{10}}{(c_1 C )^2} \cdot mq^3, \quad  n^4 > \frac{2^{13} c_2}{(c_1C)^3}\cdot m^{\frac{1}{2}} q^{5/2},
\]
whenever $n^4 \ge C' mq^3$ and $C'\sqrt{mq} > \frac{2^{13} c_2}{(c_1C)^3}$.
\item The fourth inequality is true because
\[ \frac{c_1^2K\left| \A^i \right|}{2^5n} \ge \frac{c_1^2K C n^2 c_2^{-1}m}{2^5n\sqrt{mq}} \ge \frac{c_1^2KCn\sqrt{m}}{2^5\sqrt{q}} \ge 2c_2(q-1)  \]
noting that
\[ 
K n \sqrt{m} \ge \frac{2^{5}c_2^{\frac{1}{2}}}{c_1^2 C} \cdot q^{\frac{3}{2}}.
\]
given that $m > q$ and $n \ge C^{-1} q$. 
\end{itemize}
From all the above, we conclude that given sufficiently large $q$ (for example, $q > C^{-1}2^{33}$), we can repeatedly use Lemma \ref{lem: 3.6} with $c_1 =2^{-10}$ and $c_2 =2^{20}$.
%
%
\noindent
After the process, we obtain the sequence $\A^1 \supset \A^2 \supset \cdots \supset \A^{s+1} $, with 
\[
|\G^i| \ge \frac{c_1^4 K^4 |\A^i|}{2^{11} n^2} \ge \frac{ c_1^4 K^4 c_2^{-1} m }{2^{11} n^2}. 
\]
As $\G^i$'s are disjoint, the process terminates after at most
\begin{equation*}
s\le \frac{m}{\min_i \left| \G^i \right| } \le \frac{2^{11} c_2 n^2}{c_1^4 K^4}
\end{equation*}
steps. Again, from the induction hypothesis, we have
%
\begin{equation}\label{I(A^{s+1},L)}
I(\A^{s+1},\L) < C  n^2 \sqrt{c_2^{-1}m/q} = c_2^{-\frac{1}{2}}I.     
\end{equation}
Since $\bigcup\G^i = \P \setminus (\D \cup \E \cup \A^{s+1})$, from \eqref{I(D,L)}, \eqref{I(E,L)}, and \eqref{I(A^{s+1},L)}, we are guaranteed that

\[
\sum_{i=1}^s I (\G^i, \L) = I- I(\D, \L) - I(\E, \L) -I(\A^{s+1},\L) > \left(1-c_1-2\cdot c_2^{-\frac{1}{2}}\right) I,
\]
and so,
\[ 
I(\P, \L) \le \left(1-c_1-2\cdot c_2^{-\frac{1}{2}}\right)^{-1}  \sum_{i=1}^s I(\G^i, \L).
\]
Since $I(\G^i, \L) \le K n$, we get 
\[
I \le \left(1-c_1-2\cdot c_2^{-\frac{1}{2}}\right)^{-1} s K n \le \left(1-c_1-2\cdot c_2^{-\frac{1}{2}}\right)^{-1} \cdot \frac{2^{11} c_2 n^3}{c_1^4 K^3} \lesssim \frac{n^3 m^3}{I^3},
\]
which implies that
\[
I(\P, \L) \lesssim |\P|^{\frac{3}{2}} |\L|^{\frac{3}{2}}.
\]
By Theorem \ref{multi-incidence}, we can estimate $I(\G^i, \L)$ by 
%
\[
I(\G^i, \L) 
\lesssim \frac{|\G^i||\L|}{q^2} + q^{\frac{3}{2}}\sqrt{|\G^i||\L|} \le \frac{c_2^2 K^2 n}{q^2} + q^{\frac{3}{2}} c_2 K n^{\frac{1}{2}}.
\]
Since the second term on the right-hand side always dominates when $n < q^3$, we obtain  
%
\[
I(\G^i, \L) \lesssim  q^{\frac{3}{2}} c_2 K n^{\frac{1}{2}} \le \frac{c_2 K^3 n^2}{q}.
\]
This then leads to
%
\[
I(\P, \L) \lesssim \frac{2^{11} c_2 n^2}{c_1^4 K^4} \cdot \frac{c_2 K^3 n^2}{q} = \frac{2^{11}c_2}{c_1^4} \cdot \frac{n^4}{Kq}. 
\]
Recalling that $K = I(\P, \L)/m$, this estimate suggests that
\[
I(\P, \L) < C n^2 \sqrt{m/q}.
\]

\section{Applications}

In order to demonstrate how multi-parameter results can be used, we include three applications. While all are clearly related, they each show a different facet of the multi-parameter setting.

\subsection{Dot product problem}

Erd\H os-type dot product problems have been studied in a number of contexts. Namely, given some subset of a vector space, what can we say about the distribution of dot products determined by points in our subset. Hart and Iosevich consider this problem in finite fields in \cite{HI}. One of the results they prove is the following.

\begin{theorem}[Theorem 1.4 from \cite{HI}]
    Suppose $\mathcal E\subset \mathbb F_q^d,$ with $|\mathcal E|\gtrsim q^\frac{d+1}{2},$ then
    \[\mathbb F_q^* \subseteq \{x\cdot y:x,y\in \mathcal E\}.\]
\end{theorem}

To prove this, they get bounds on how often individual dot products occur. Their key estimate is for any nonzero $a\in \mathbb F_q,$

\begin{equation}\label{dpKey}
D(\mathcal E,a)\coloneqq |\{(x,y)\in \mathcal E\times \mathcal E : x\cdot y = a\}| \lesssim \frac{|\mathcal E|^2}{q} + q^\frac{d-1}{2}|\mathcal E|.
\end{equation}

In this section, we prove an analogous result in multi-parameter setting.

\begin{theorem}\label{thm:multi-dot}
Let $\E$ be a subset of $\F_q^2 \times \F_q^2$. For any $a,b \in \F_q$, we have 
\[
\mathbf{D}(\E, a, b) \coloneqq \left|\left\{((x, y), (z, t)) \in \E \times \E \colon x \cdot z = a, \ y \cdot t = b\right\}\right| \lesssim  \frac{|\E|^2}{q^2} + q^{\frac{3}{2}}|\E|.
\]
Moreover, if $|\E| \gtrsim q^{\frac{7}{2}}$, then we have
\[
\mathbf{D}(\E, a, b) \lesssim q^{-2}|\E|^2.
\]
\end{theorem}

\begin{proof}[Proof of Theorem \ref{thm:multi-dot}]
Fix $a,b \in \F_q$. Consider the following multi-collection of line-pairs 
\[
\L \coloneqq \{((u \cdot X = a), (v \cdot Y = b)) \colon (u, v) \in \E\}.
\]
Then, it is obvious that $\mathbf{D}(\E, a, b)$ is exactly $I(\E, \L)$ and $|\L| = |\E|$. By Theorem \ref{multi-incidence}, we have 
\[
I(\E, \L) \lesssim \frac{|\E||\L|}{q^2} + q^{\frac{3}{2}}\sqrt{|\E||\L|} = \frac{|\E|^2}{q^2} + q^{\frac{3}{2}}|\E|.
\]
When $|\E|\gtrsim q^{\frac{7}{2}}$, we further have 
\[
\frac{|\E|^2}{q^2} \gtrsim q^{\frac{3}{2}}|\E|,
\]
which leads to
\[
I(\E, \L) \lesssim \frac{|\E|^2}{q^2},
\]
as desired.
\end{proof}

Given $\E \subseteq \F_q^4$, we note that
\[
|\{x, y \in \E \colon x \cdot y = t\}| = \sum_{a+b = t} \left|\left\{x, y \in \E \colon x_1 y_1 + x_2 y_2 = a, \ x_3 y_3 + x_4y_4 = b\right\}\right|.
\]
As a result, we obtain the following estimate
\[
|\{x, y \in \E \colon x \cdot y = t\}| \lesssim \frac{|\E|^2}{q} + q^{\frac{1}{2}} |\E|.
\]

\subsection{Sum-product estimate}
One of the most-studied problems in additive combinatorics is the sums and products problem. One version of the problem starts by considering any large finite set of integers, $\A$. The goal is to show that either the set of pairwise sums of $\A$ or the set of pairwise products of $\A$ must be much larger than the size of $\A$. To state this precisely, we define the sumset and product set, respectively, 
\[
\A+\A\coloneqq\{a+b\colon a,b\in \A\} \text{ and } \A\A\coloneqq\{ab\colon a,b\in \A\}.
\]
The conjecture of Erd\H{o}s and Szemer\'{e}di is that $\max\{|\A+\A|,|\A\A|\}\gtrsim |\A|^{2-\epsilon},$ for any $\epsilon>0.$ For more on this and related problems, see the books by Nathanson \cite{N} or Tao and Vu \cite{TV}. In the specific context of finite fields, there has been much activity, much of which stemmed from the work of Bourgain, Katz, and Tao \cite{BKT}, and gaining further attention with papers like \cite{HIS} by Hart, Iosevich, and Solymosi.

In the present context, we want to deal with pairs of field elements, so for $x, y \in \F_q^2$, let us denote the element-wise product of $x$ and $y$ by $x \otimes y$. Specifically, we have
\[
\begin{pmatrix}
x_1\\x_2
\end{pmatrix}
\otimes
\begin{pmatrix}
y_1\\y_2
\end{pmatrix}
=
\begin{pmatrix}
x_1y_1\\x_2y_2
\end{pmatrix}
.
\]
Given this multi-parametric definition of product, we present a sum-product type result.
\begin{theorem}\label{thm:multi-sum-product}
Let $\A$ be a set of points in $\F_q^2$. Assume that 
\[
\min \{|\A+\A|, |\A \otimes \A| \} \lesssim q^{5}|\A|^{-2} .
\]
Then, we have
\[
\max\{|\A+\A|, |\A \otimes \A|\} \gtrsim q^{-\frac{3}{2}}|\A|^2.
\]
\end{theorem}
\begin{remark}
When $|\A| \sim q^{5/3}$, the result says that if 
\[
\min \{|\A+\A|, |\A \otimes \A| \} \lesssim |\A|,
\]
then
\[
\max \{|\A+\A|, |\A \otimes \A| \} \gtrsim |\A|^\frac{11}{6}.
\]
\end{remark}

\begin{proof}[Proof of Theorem \ref{thm:multi-sum-product}]
Set $m = \min \{|\A+\A|, |\A \otimes \A| \}$. Let us define
\[
\P \coloneqq |\A+\A| \times |\A \otimes \A|, 
\]
\[
\L \coloneqq \{(\{Y = a_1(X-a_2)\}, \{Y = a'_1(X-a'_2)\}) \colon (a_1, a_2, a'_1, a'_2) \in \A \times \A\}.
\]
Since $m|\A|^2 \lesssim q^5$, an application of Theorem \ref{AtimesB,AL_small} for $\P$ and $\L$ then yields
\begin{equation}\label{sumproduct:1}
I(\P, \L) \lesssim q^{\frac{3}{2}}\sqrt{|\P||\L|} = q^{\frac{3}{2}}|\A|\sqrt{|\A+\A||\A \otimes \A|}.    
\end{equation}
Since the pair of lines $y = a_1(x-a_2)$ and $y = a'_1(x-a'_2)$ contains the point $(C + a_2, a_1C, C' + a'_2, a'_1C')$ for any choice of $(C, C') \in \A$, each of the $|\A|^2$ lines contributes at least $|\A|$ incidences, and so we have
\begin{equation}\label{sumproduct:2}
|\A|^3 \le I(\P, \L).
\end{equation}
By combining \eqref{sumproduct:1} and \eqref{sumproduct:2}, we obtain
\[
|\A+\A||\A \otimes \A| \gtrsim q^{-3}|\A|^4,
\]
which leads to
\[
\max\{|\A+\A|, |\A \otimes \A|\} \gtrsim q^{-\frac{3}{2}}|\A|^2,
\]
as desired.
\end{proof}

\subsection{Vector-valued functions}
We offer one final example application. One can combine the ideas from the previous two subsections to get results like the following.

\begin{theorem}\label{multi-vector-valued}
Let $\A$ and $\B$ be two subsets of $\F_q^2$. Consider the function \[F(x,y) = (x_1^2-x_1y_1, x_2^2-x_2y_2).\]
If $|\A|^2|\B| \lesssim q^{5}$, then we have
\[
|F(\A, \B)| \gtrsim q^{-\frac{3}{2}} |\A||\B|.
\]
\end{theorem}
\begin{proof}[Proof of Theorem \ref{multi-vector-valued}]
First, let us define the following set
\[
\E \coloneqq \{(a, b, a', b') \in \A \times \B \times \A \times \B \colon F(a, b) = F(a', b')\}.
\]
Then, a direct application of Lemma \ref{lemma7-SdZ} gives
\begin{equation}\label{proof:mvv,eq1}
|\A||\B| \le |F(\A, \B)|^{\frac{1}{2}} |\E|^{\frac{1}{2}}.
\end{equation}
On the other hand, we can bound $|\E|$ using Theorem \ref{AtimesB,AL_small}, by considering a solution of the system 
\[
\begin{cases}
a_1^2 - a_1b_1 = {a'_1}^2 - a'_1b'_1\\
b_2^2 - a_2b_2 = {b'_2}^2 - a'_2b'_2
\end{cases}
\]
as an incidence between the point $(b, b')$ and the line-pair
\[
\ell_{a, a'} \coloneqq (\{a_1^2 - a_1X = {a'_1}^2 - a'_1Y\}, \{a_2^2 - a_2X = {a'_2}^2 - a'_2Y\}).
\]
Consider the point set and the line-pair set
\[
\P \coloneqq \B \times \B, 
\quad
\L \coloneqq \{\ell_{a,a'} \colon a_1 \neq \pm a'_1 ~\mbox{or}~ a_2 \neq \pm a'_2\}.
\]
By the observation made earlier, it is straightforward to check that $|\E| \le 2|I(\P, \L)|$. By applying Theorem \ref{AtimesB,AL_small} with the two sets above, we obtain that 
\begin{equation}\label{proof:mvv,eq2}
|\E| \le 2|I(\P, \L)| \lesssim q^{\frac{3}{2}} |\A||\B|,
\end{equation}
provided that $|\A|^2|\B| \lesssim q^{5}$. Combining \ref{proof:mvv,eq1} and \ref{proof:mvv,eq2} gives the estimate
\[
|F(\A, \B)| \gtrsim q^{-\frac{3}{2}} |\A||\B|,
\]
whenever we have $|\A|^2|\B| \lesssim q^{5}$.
\end{proof}

\pagebreak
\section{Acknowledgements}
We would like to thank to the Vietnam Institute for Advanced Study in Mathematics for the hospitality and for the excellent working conditions, and Thang Pham for suggesting we investigate multi-parameter problems.

\end{document}